\documentclass[a4paper, 10pt]{amsart}
\usepackage[utf8]{inputenc}
\usepackage[T1]{fontenc}
\usepackage{lmodern}
\usepackage[english]{babel}
\usepackage[dvipsnames]{xcolor}
\usepackage{adjustbox}
\usepackage{letltxmacro}
\usepackage{todonotes}
\LetLtxMacro\todonotestodo\todo
\renewcommand{\todo}[2][]{\todonotestodo[#1]{TODO: {#2}}}
\usepackage{enumerate}
\usepackage{hyperref}
\usepackage{url}
\usepackage{tikz-cd}
\usepackage{verbatim}
\newtheorem{theorem}{Theorem}

\makeatletter
\newtheorem*{rep@theorem}{\rep@title}
\newcommand{\newreptheorem}[2]{%
\newenvironment{rep#1}[1]{%
 \def\rep@title{#2 \ref{##1}}%
 \begin{rep@theorem}}%
 {\end{rep@theorem}}}
\makeatother

\newreptheorem{theorem}{Theorem}
\newtheorem{theorem:intro}{Theorem}

\newtheorem{lemma}{Lemma}[section]
\newtheorem{proposition}[lemma]{Proposition}
\newtheorem{corollary}[lemma]{Corollary}

\theoremstyle{definition}
\newtheorem{remark}[lemma]{Remark}

\newtheorem*{claim*}{Claim}

\newtheorem*{theorem*}{Theorem}
\newtheorem*{corollary*}{Corollary}
\newtheorem*{lemma*}{Lemma}

\newcommand{\Z}{\mathbb{Z}}
\newcommand{\N}{\mathbb{N}}

\DeclareMathOperator{\spec}{spec}
\DeclareMathOperator{\Reg}{Reg}
\DeclareMathOperator{\im}{im}
\DeclareMathOperator{\Pic}{Pic}
\newcommand{\vmax}{v\text{-}\max}

\title{On the distribution of prime divisors in class groups of affine monoid algebras}
\author{Victor Fadinger}
\address[Victor Fadinger]{Institut für Mathematik und Wissenschaftliches Rechnen\\Karl-Franzens-Universität Graz\\
  Heinrichstraße 36\\8010 Graz\\Austria}
\email{\href{mailto:victor.fadinger@uni-graz.at}{victor.fadinger@uni-graz.at}}
\thanks{V.~Fadinger is supported by the Austrian Science Fund (FWF): W1230}

\author{Daniel Windisch}
\address[Daniel Windisch]{Institut für Analysis und Zahlentheorie\\Technische Universität Graz\\
  Kopernikusgasse 24/II\\8010 Graz\\Austria}
\email{\href{mailto:dwindisch@math.tugraz.at}{dwindisch@math.tugraz.at}}
\thanks{D.~Windisch is supported by the Austrian Science Fund (FWF): P~30934}


\begin{document}

\begin{abstract}
We investigate the class groups and the distribution of prime divisors in affine monoid algebras over fields and thereby extend the result of Kainrath \cite{Kain} that every finitely generated integral algebra of Krull dimension at least $2$ over an infinite field has infinitely many prime divisors in all classes.
\end{abstract}

\maketitle

\section{Introduction}

The study of class groups and of the distribution of prime divisors in the classes is an old topic in ring theory. In general, class groups and the distribution of prime divisors can more or less be arbitrary. Indeed, for every abelian group $G$ and every subset $G_0 \subseteq G$, which generates $G$ as a monoid, there is a Dedekind domain $R$ whose class group is isomorphic to $G$ and $G_0$ corresponds to the set of classes containing prime divisors (this is Claborn's Realization Theorem, see \cite[Theorem 3.7.8]{GHK}). It is classical that rings of integers in algebraic number fields and holomorphy rings in algebraic function fields are Dedekind domains with finite class group and infinitely many prime divisors in all classes. More recent realization results show that every abelian group is isomorphic to the class group of a simple Dedekind domain (\cite{Sm}) and isomorphic to the class group of a ring of Krull type that is not Krull (\cite{Chang-Krulltype}).  Furthermore, monoid algebras that are Krull and cluster algebras that are Krull do have infinitely many prime divisors in all classes (\cite{Krull} and \cite[Theorem A]{Ga-La-Sm19a}). \\
Concerning Noetherian domains, that are not necessarily Krull (equivalently, not integrally closed) we mention a result by Kainrath \cite{Kain}. He proved that for an infinite field $K$, every finitely generated integral $K$-algebra with quotient field seperable over $K$ and Krull dimension at least $2$ has infinitely many prime divisors in each class. So in particular, this holds true for finitely generated monoid algebras of this type. \\
The distribution of prime divisors is still open in the following cases:
\begin{enumerate}
\item finitely generated algebras over arbitrary fields of Krull dimension one and
\item finitely generated algebras over finite fields of Krull dimension at least two.
\end{enumerate}
In the present paper, we study these two cases when the algebra is a monoid algebra. More precisely, in the one-dimensional case we obtain the following sufficient condition.

\begin{theorem:intro}
Let $K$ be a field and let $S$ be a numerical monoid with Frobenius number $\mathsf{f}(S)=\max(\N_0\setminus S)$. Then $K[S]$ has at least one prime divisor (resp. infinitely many) in all classes if for all $a_0,\ldots ,a_{\mathsf{f}(S)}\in K$ with $a_0\neq 0$ there exists at least one irreducible polynomial (resp. infinitely many) in $K[X]$ whose coefficient at the monomial $X^i$ equals $a_i$ for all $i\in[0,\mathsf{f}(S)]$.
\end{theorem:intro}

\noindent
To achieve this, we in addition give descriptions of class groups of such monoid algebras.
In the higher dimensional case, we have the following theorem for arbitrary fields.

\begin{theorem:intro}
Let $K$ be a field and let $S$ be an affine monoid with factorial complete integral closure and quotient group of rank at least $2$. Then $K[S]$ has infinitely many prime divisors in each divisor class.
\end{theorem:intro}

\section{Preliminaries} \label{section:preliminaries}

\noindent
\textbf{Monoids.} In our setting, a \textit{monoid} $S$ is an additively written commutative cancellative semigroup with identity element. We denote its quotient group by $\mathsf{q}(S)$. \\
An \textit{affine monoid} is finitely generated monoid with torsion-free quotient group. Equivalently, it is isomorphic to a finitely generated submonoid of $(\Z^m,+)$ for some $m \in \N$. A \textit{numerical monoid} is a submonoid $S$ of $(\N_0,+)$ such that $\N_0 \setminus S$ is finite and $\mathsf{f}(S) = \max(\N_0\setminus S)$ denotes its \textit{Frobenius number}. Note that the affine monoids with quotient group $\Z$ are exactly the numerical monoids or $\Z$ itself (up to isomorphism).\\
For a monoid $S$ with quotient group $G = \mathsf{q}(S)$, we denote by $\widehat{S} = \{x \in G \mid \exists c \in S \ \forall n \in \N \ c + nx \in S \}$. This is again a monoid with quotient group $G$ and called the \textit{complete integral closure} of $S$. Moreover, we call $\mathfrak{f}_S = \{g \in G \mid g + \widehat{S} \subseteq S \}$ the \textit{conductor} of $S$, which is an ideal of $S$ and of $\widehat{S}$. \\

\noindent
\textbf{Monoid algebras.} For a domain $D$ and an additively written monoid $S$, let $D[S]$ denote the monoid algebra of $D$ over $S$. It is well known that the monoid algebra $D[S]$ is a domain if and only if the monoid $S$ is torsion-free \cite[Theorem 8.1]{Gil84}. Therefore, the finitely generated monoid algebras that are domains are exactly the affine monoid algebras. \\
The Krull dimension of a monoid algebra $D[S]$ as above is the same as that of $D[\bf X]$, where $\bf X$ is a set of indeterminates over $D$ with cardinality equal to the torsion-free rank of the quotient group of $S$ \cite[Theorems 17.1 \& 21.4]{Gil84}. In particular, if $K[S]$ is a finitely generated monoid algebra over a field $K$, then its Krull dimension equals $1$ if and only if $S$ is isomorphic to $\Z$ or to a numerical monoid.\\
For subsets $P\subseteq D$ and $H\subseteq S$, we denote by $P[H]=\{\sum_{i=1}^{n}p_iX^{h_i}\mid n\in\N, p_i\in P, h_i\in H\}\subseteq D[S]$ all elements of $D[S]$ with coefficients in $P$ and exponents in $H$.\\
Let $K$ be the quotient field of $D$ and let $G$ be the quotient group of $S$. For an element $f = \sum_{i=1}^n d_i X^{s_i} \in K[G]$ with $d_i \in K$ and $s_i \in G$ we denote by $A_f$ resp. $E_f$ the fractional ideal of $D$ resp. $S$ generated by $d_1,\ldots,d_n$ resp. $s_1,\ldots,s_n$.\\

\noindent
\textbf{$v$-ideals.} Let $S$ be a monoid with quotient group $G$. For a subset $X \subseteq G$, we set $X^{-1} =(S:X)= \{g \in G \mid g+X \subseteq S\}$ and $X_v = (X^{-1})^{-1}$. A subset $I \subseteq G$ is called a \textit{fractional $v$-ideal} if there exists $c \in S$ such that $c + I \subseteq S$ and $I = I_v$ and it is called a (\textit{integral}) $v$-\textit{ideal} if $I \subseteq S$ and $I = I_v$. A fractional $v$-ideal $I$ is called $v$-\textit{invertible} if there exists a fractional $v$-ideal $J$ such that $(IJ)_v = S$. A $v$-invertible fractional $v$-ideal $I \subseteq S$ is called a \textit{$v$-invertible $v$-ideal}. We denote the set of all $v$-invertible $v$-ideals by $\mathcal{I}_v^*(S)$. The set of all $v$-invertible fractional $v$-ideals with multiplicative defined as $I \cdot_v J = (IJ)_v$ forms a group with identity element $S$.
It is a fact that every (fractional) $v$-ideal of $S$ is indeed a (fractional) ideal of $S$ \cite[Section 11.4]{HK98}. \\
Replacing the monoid $S$ by a domain $D$, the terminology and notation concerning $v$-ideals above is the same. \\
A domain $D$ is called a \textit{Mori domain} if it satisfies the ascending chain condition on $v$-ideals.\\

\noindent
\textbf{Class groups.} By $\mathcal{C}_v(S)$ we denote the $v$-class group of $S$ which is the factor group of the group of $v$-invertible fractional $v$-ideals of $S$ modulo the subgroup of non-empty principal fractional ideals of $S$. It is known that $\mathcal{C}_v(S)$ is isomorphic to the quotient of the monoid of $v$-invertible (integral) $v$-ideals modulo the submonoid of non-empty principal (integral) ideals. \\
If $I$ is a $v$-invertible fractional $v$-ideal of $S$, we denote by $[I] \in \mathcal{C}_v(S)$ its $v$-class and refer to it as the class of $I$. Moreover, we denote by $\mathfrak{X}(S)$ the set of height-one prime ideals of $S$. In the literature, the $v$-invertible elements of $\mathfrak X(S)$ are also often called prime divisors of $S$.\\
Replacing the monoid $S$ by a domain $D$, the terminology and notation concerning class groups and prime divisors above is the same. For further information, see \cite[Section 2.10]{GHK}.\\

\noindent
\textbf{The monoid of regular elements.} For a domain $D$ we define its \textit{monoid of regular elements} as $\Reg(D) = \{a \in D \mid \text{ if } z \in \widehat{D} \text{ and } za \in D \text{, then } z \in D\}$. For a Mori domain $D$, one can easily see that $\Reg(D)$ is a saturated submonoid of $(\widehat D \setminus \{0\}, \cdot)$ as follows: Let $a,b \in \Reg(D)$ and $z \in \widehat D \setminus \{0\}$ such that $az = b$. Since $b \in D \setminus \{0\}$, it follows by the definition of $\Reg(D)$ that $z \in D$. Since $\Reg(D)$ is a divisor closed submonoid of $D \setminus \{0\}$ by \cite[Proposition 2.3.10.2]{GHK} we have that $z \in \Reg(D)$.\\
For a domain $D$ with complete integral closure $\widehat D$, we denote by $\mathcal{I}_\mathfrak{f}(D)$ the set of all $v$-invertible $v$-ideals $I$ with $I \cap \Reg(D) \neq \emptyset$ and by $\mathcal{I}_\mathfrak{f}(\widehat D)$ the set of $v$-invertible $v$-ideals $I$ of $\widehat{D}$ with $I \cap \Reg(D) \neq \emptyset$.\\
By \cite[Theorem 2.10.9]{GHK}, for a Mori domain with $\mathfrak f= (D:\widehat D) \neq (0)$ we have isomorphisms
\begin{align*}
\hspace{3cm} \varphi: \mathcal{I}_\mathfrak{f}(\widehat{D}) &\to \mathcal{I}_\mathfrak{f}({D}) \hspace{2.5cm} &\psi: \mathcal{I}_\mathfrak{f}({D}) &\to \mathcal{I}_v^*(\Reg(D)) \hspace{3cm} \\
I &\mapsto I \cap D, &{} I &\mapsto I \cap \Reg(D).
\end{align*}
The isomorphism $\psi$ induces an isomorphism $\tilde{\psi}: \mathcal{C}_v(D) \to \mathcal{C}_v(\Reg(D))$. In particular, for every $g \in \mathcal{C}_v(D)$ there is a bijective map from the prime divisor in $g$ to the prime divisors in $\tilde{\psi}(g)$ given by intersection with $\Reg(D)$.\\

\noindent
\textbf{Valuations.} In Proposition \ref{proposition:use-kainrath}, we make use of the theory of valuations on groups. If $S$ is a monoid and $P$ is a prime ideal of $S$ such that the localization $S_P$ is a valuation monoid, then we denote by $\mathsf{v}_P$ its valuation induced on the quotient group of $S$ and call it the $P$-adic valuation. The interested reader is referred to \cite[§ 16]{Gil84}. \\

\section{The one-dimensional case} \label{section:one-dim}

\begin{lemma}\label{lemma:conductor_general}
Let $D$ be an integral domain and let $S$ be a torsion-free monoid. Then the conductor $\mathfrak{f}_{D[S]}$ equals $\mathfrak{f}_D[\mathfrak{f}_S] = \{\sum_{i=1}^{n}p_iX^{h_i}\mid n\in\N, p_i\in \mathfrak{f}_D, h_i\in \mathfrak{f}_S\}$.
\end{lemma}

\begin{proof}
$\mathfrak{f}_D[\mathfrak{f}_S] \subseteq \mathfrak{f}_{D[S]}$ is obvious. Now, let $g \in \mathfrak{f}_{D[S]}$ and write $g = \sum_{i = 1}^n d_i X^{s_i}$ with $d_i \in D$ and $s_i \in S$ for all $i \in [1,n]$. Clearly, we have $X^\alpha g = \sum_{i =1}^n d_i X^{s_i + \alpha} \in D[S]$ for all $\alpha \in \widehat{S}$. Thus, $s_i + \alpha \in S$ for all $i \in [1,n]$ and $\alpha \in \widehat{S}$. It follows that $s_i \in \mathfrak{f}_S$ for all $i \in [1,n]$.  In the same way, we have $d g = \sum_{i =1}^n (dd_i) X^{s_i} \in D[S]$ for all $d \in \widehat{D}$. Thus, $ dd_i \in D$ for all $i \in [1,n]$ and $d \in \widehat{D}$. It follows that $d_i \in \mathfrak{f}_D$ for all $i \in [1,n]$. 
\end{proof}

\begin{lemma}\label{lemma:conductor}
Let $K$ be a field, let $S$ be a numerical monoid, and let $\mathfrak{f}_S = (S:\widehat{S})$ resp. $\mathfrak{f} = (K[S]: \widehat{K[S]})$ denote the conductor of $S$ resp. $K[S]$. 
\begin{enumerate}
\item $\widehat{S}= \N_0$ and $\widehat{K[S]} = K[X]$.
\item $\mathfrak{f}_S = [\mathsf{f}(S)+1, \infty)$ and $\mathfrak{f} = K[\mathfrak{f}_S]$, where $\mathsf{f}(S) = \max (\N_0 \setminus S)$.
\item $\Reg(K[S]) = \{f \in K[S] \mid f(0) \neq 0 \}$.
\end{enumerate}
\end{lemma}

\begin{proof}
(1) The first equality is clear and the second follows from $\widehat{K[S]} = K[\widehat{S}]$ (see the proof of \cite[Theorem 12.5]{Gil84}). \\
(2) The first equality is trivial. The second one follows from Lemma \ref{lemma:conductor_general}.
(3) We have
\begin{align*}
\mathfrak{X}(K[S]) = \{fK[\Z] \cap K[S] \mid f \in K[\Z] \text{ irreducible} \} \cup \{K[S\setminus \{0\}] \}
\end{align*}
by \cite[Lemma 4.7]{FW} and the fact that $S \setminus \{0\}$ is the only non-empty prime ideal of $S$. But the only prime ideal among these containing $\mathfrak{f} = K[\mathfrak{f}_S]$ is $K[S \setminus \{0\}]$, whence by \cite[Theorem 2.10.9.2]{GHK} and the fact that $\vmax(K[S]) = \mathfrak{X}(K[S])$ (note that $K[S]$ is one-dimensional Noetherian by \cite[Theorems 7.7, 17.1, and 21.4]{Gil84}) we have
\begin{align*}
\Reg(K[S]) &= K[S] \setminus K[S \setminus \{0\}] \\
			& = \{f \in K[S] \mid f(0) \neq 0 \}.
\end{align*}
\end{proof}

Let $\mathcal{C} = \{1+ \sum_{i =1}^{\mathsf{f}(S)} k_iX^i \mid k_i \in K\} \subseteq K[S]$ and endow $\mathcal{C}$ with the group operation $\otimes$ defined by multiplying two polynomials in $\mathcal{C}$ and setting the coefficients of monomials with exponents $>m$ in the product equal to $0$, that is, $(1+ \sum_{i =1}^{\mathsf{f}(S)} k_iX^i) \otimes (1+ \sum_{i =1}^{\mathsf{f}(S)} k_iX^i) = 1+ \sum_{i =1}^{\mathsf{f}(S)} (\sum_{a+b = i} k_al_b) X^i$ with $k_0 = l_0 = 1$.

\begin{lemma}\label{lemma:general}
Let $K$ be a field and let $S$ be the numerical monoid
%such that $\N_0 \setminus S = [1,m]$
. Denote $\mathsf{R} = \Reg(K[S])$ and $\Lambda=\N_0\setminus S$.
\begin{enumerate}
\item $\{f \in K[X] \mid fK[X]  \in \mathcal{I}_\mathfrak{f}(K[X])\} = \{f \in K[X] \mid f(0) \neq 0 \}$.
\item For $f,g \in K[X]$ with $fK[X], gK[X] \in \mathcal{I}_\mathfrak{f}(K[X])$ we have $[fK[X] \cap \mathsf{R}] = [gK[X] \cap \mathsf{R}]$ if and only if $\frac{f}{g} \in \mathsf{q(R)} $.
\end{enumerate}
\end{lemma}

\begin{proof}
(1) Let $f \in K[X]$ with $fK[X] \in \mathcal{I}_\mathfrak{f}(K[X])$, i.e., $fK[X] \cap \mathsf{R} \neq \emptyset$. It follows from Lemma \ref{lemma:conductor}(3) that $fg \in K[S]$ for some $g \in K[X]$ with $(fg)(0) \neq 0$. Therefore $f(0) \neq 0$. \\
For the converse inclusion, let $f \in K[X]$ with $f(0) \neq 0$. Without loss of generality, we may assume that $f = 1 + \sum_{i = 1}^n k_i X^i$. We recursively construct $g = 1 + \sum_{j = 1}^{\max(\Lambda)} l_j X^j$ such that $fg \in \mathsf{R}$, that is, $fg \in K[S]$ and $(fg)(0) \neq 0$.\\
We set $l_1= -k_1$. If now $j >1$, we set $l_j = - \sum_{\substack{a+b = j \\ b \neq j}} k_al_b$. \\
We have $(fg)(0) = 1 \neq 0$. Now for $r \in \Lambda$, the coefficient of the monomial $X^r$ in $fg$ is $\sum_{a+b = r} k_al_b$, where $k_0 = l_0 = 1$. Since $l_r = - \sum_{\substack{a+b = r \\ b \neq r}} k_al_b$, we see that this coefficient equals $0$ and hence $fg \in K[S]$. \\
(2) Let $f,g \in K[X]$ with $fK[X],gK[X] \in \mathcal{I}_\mathfrak{f}(K[X])$. Assume first that $[fK[X] \cap \mathsf{R}] = [gK[X] \cap \mathsf{R}]$. It follows that there exists $a,b \in \mathsf{R}$ such that $a(fK[X] \cap \mathsf{R}) = b(gK[X] \cap \mathsf{R})$. As noted in the preliminaries, we have an isomorphism
\begin{align*}
\delta: \mathcal{I}_\mathfrak{f}(K[X]) &\to \mathcal{I}_v^*(\mathsf{R}), \\
I &\mapsto I \cap \mathsf{R}.
\end{align*}
Now $aK[X] \cap \mathsf{R} = a\mathsf{R}$ and $bK[X] \cap \mathsf{R} = b\mathsf{R}$, because $\mathsf{R}$ is a saturated submonoid of $K[X]\setminus \{0\}$. It follows that $afK[X] = \delta^{-1}(a(fK[X] \cap \mathsf{R})) = \delta^{-1}(b(gK[X] \cap \mathsf{R})) = bgK[X]$. So $af$ and $bg$ are associates in $K[X]$, whence there exists some $k \in K\setminus \{0\}$ such that $kaf = bg$. It follows that $\frac{f}{g} = \frac{ka}{b} \in \mathsf{q(R)}$, because $K\setminus \{0\} \subseteq \mathsf{R}$. \\
For the converse direction, assume that $\frac{f}{g} \in \mathsf{q(R)}$. Then there exist $a,b \in \mathsf{R}$ such that $\frac{a}{b} = \frac{f}{g}$. Now, by the same argument as above, we have $aK[X] \cap \mathsf{R} = a\mathsf{R}$ and $bK[X] \cap \mathsf{R} = b\mathsf{R}$. Therefore $b(fK[X] \cap \mathsf{R}) = \delta(bfK[X]) = \delta(agK[X]) = a(gK[X] \cap \mathsf{R})$.\\

\end{proof}

\begin{proposition}\label{proposition:classgroup1}
Let $K$ be a field, $m \in \N$, and let $S$ be the numerical monoid such that $\N_0 \setminus S = [1,m]$
. Denote $\mathsf{R} = \Reg(K[S])$.
\begin{enumerate}
\item We have $\Pic(K[S]) = \mathcal{C}_v(K[S])$ and there is an isomorphism $\phi: \mathcal{C}_v(K[S]) \to (\mathcal{C}, \otimes)$ sending $[fK[X] \cap \mathsf{R}]$ with $f = 1 + \sum_{i =1}^n k_iX^i$ to $ 1 + \sum_{i =1}^m k_iX^i$ (filling up with zeros if $n < m$).
\item The prime divisors in a class $c \in \mathcal{C}_v(K[S])$ are in one-to-one correspondence with the irreducible polynomials $f \in K[X]$ such that $f = 1+ \sum_{i =1}^n k_iX^i$ and $\phi(c) = 1+ \sum_{i =1}^m k_iX^i$ (filling up with zeros if $n<m$).
\item The following are equivalent.
	\begin{itemize}
	\item[(a)] Every divisor class of $K[S]$ contains (infinitely many) prime divisors.
	\item[(b)] For all $a_0,\ldots ,a_m\in K$ with $a_0\neq 0$ there exist (infinitely many) irreducible polynomials in $K[X]$ whose coefficient at the monomial $X^i$ equals $a_i$ for all $i\in[0,m]$.
	\end{itemize}
\end{enumerate}
\end{proposition}
\begin{proof}
(1) The equality $\Pic(K[S]) = \mathcal{C}_v(K[S])$ follows from \cite[Proposition 2.10.5.1]{GHK} and the fact that $K[S]$ is one-dimensional Noetherian.\\
To see that $\phi$ is well-defined, let $f,g \in K[X]$ with $f = 1 + \sum_{i=1}^nk_iX^i$ and $g = 1 + \sum_{j = 1}^rl_jX^j$ such that $[fK[X] \cap \mathsf{R}] = [gK[X] \cap \mathsf{R}]$. By (2), we have that $\frac{f}{g} \in \mathsf{q(R)}$. It follows that $fb = ga$ for some $a,b \in \mathsf{R}$. Since $b(0) = (fb)(0) = (ga)(0) = a(0)$, the constant terms of $a$ and $b$ are equal and non-zero, because $a,b \in \mathsf{R}$. It follows that $1 + \sum_{i=1}^mk_iX^i = 1 + \sum_{j = 1}^ml_jX^j$.\\
It is clear by using the isomorphism $\delta$ from above that $\phi$ is a homomorphism, and it is surjective by definition.\\
The injectivity follows by the same construction as in the proof of Lemma \ref{lemma:general} (1). Namely, if $f,f' \in K[X]$ with $f(0) = f'(0) = 1$ and $\phi([fK[X] \cap \mathsf{R}]) = \phi([fK[X] \cap \mathsf{R}])$, then the coefficients of $f$ and $f'$ agree for indices up to $m$ and by the construction in Lemma \ref{lemma:general} (1), we get one single polynomial $g \in K[X] \setminus \{0\}$ with $fg,f'g \in \mathsf{R}$. Therefore $\frac{f}{f'} \in \mathsf{q(R)}$. \\
(2) This follows immediately from (1) and the fact that the prime divisors of $\mathsf{R}$ are exactly the ideals $fK[X]\cap\mathsf R$ for $f\in K[X]$ irreducible and $f(0)\neq 0$.\\
(3) This is obvious by (2).
\end{proof}

In the following we will make use of a known result. For the convenience of the reader we formulate it here.

\begin{remark}\label{remark:known}
\cite[Proposition 8.9.7]{GHK} Let $K$ be a one-dimensional function field over a finite field and let $\emptyset\neq M\subseteq \mathcal P(K)$ be a non-empty finite subset of the set of discrete rank 1 valuation domains of $K$. Moreover, let $A\subseteq R_M:=\bigcap_{V\in \mathcal P(K)\setminus M} V$ be an order in the holomorphy ring $R_M$ of $M$. Then every class of $\Pic(A)=\mathcal C_v(A)$ contains infinitely prime divisors.
\end{remark}

It is well-known, that in polynomial rings in one variable over finite fields there are irreducible polynomials with prescribed coefficients. We can reobtain this result as a consequence of the previous remark and Proposition \ref{proposition:classgroup1}.

\begin{corollary}\cite[Section 4.2]{Pollack}\label{corollary:irreducibles}
Let $K$ be a finite field and let $Y$ be an indeterminate over $K$. Moreover, let $a_0,\ldots ,a_m\in K$ with $a_0\neq 0$. Then for infinitely many tuples $(n,a_{m+1},\ldots ,a_n) \in \N_{\geq m} \times K^{n-m}$ the polynomial $\sum_{i=0}^n a_iY^i$ is irreducible.
\end{corollary}
\begin{proof}
Let $a_0,\ldots ,a_m\in K$ with $a_0\neq 0$ and let $S$ be the numerical monoid with $\N_0\setminus S=[1,m]$. If we can show, that $K[S]$ has infinitely many prime divisors in every divisor class, we are done by Proposition \ref{proposition:classgroup1}. $K[X]$ is a holomorphy ring, as $K[X]=\bigcap_{\substack{f\in K[X]\\ f \text{ irreducible}}} K[X]_{(f)}$ is the intersection of all discrete rank 1 valuation domains of $K(X)$ except the one stemming from the valuation defined by $\frac{f}{g}\mapsto \deg(f)-\deg(g)$. Since the quotient fields of $K[S]$ and $K[X]$ coincide and $K[X]$ is a finitely generated $K[S]$-module, $K[S]$ is an order in $K[X]$ and the assertion follows from Remark \ref{remark:known}.
\end{proof}

\begin{proposition}\label{proposition:classgroup2}
Let $K$ be a field and let $S\subseteq T\subseteq \N_0$ be numerical monoids. Then there exists an epimorphism of class groups
\begin{align*}
\Theta:\mathcal C_v(K[S])&\to\mathcal C_v(K[T]),\\
[fK[X]\cap K[S]]&\mapsto [fK[X]\cap K[T]]
\end{align*}
that maps classes containing (infinitely many) prime divisors to classes containing (infinitely many) prime divisors and whose kernel consists of all those classes that contain $fK[X]\cap K[S]$ for some $f\in K[T]$ with $f(0)\neq 0$.
\end{proposition}
\begin{proof}
It follows from \cite[Theorem 2.10.9.5]{GHK} and Lemma \ref{lemma:conductor} that $\eta:\mathcal I_v^*(\Reg(K[S]))\cong \mathcal I_{\mathfrak f_S}(K[X])=\mathcal I_{\mathfrak f_T}(K[X])\cong \mathcal I_v^*(\Reg(K[T]))$ with $\eta(fK[X]\cap \Reg(K[S]))=fK[X]\cap \Reg(K[T])$ is an isomorphism. Note, that $fK[X]\cap \Reg(K[S])$ is a principal ideal of $\Reg(K[S])$ if and only if $f\in \Reg(K[S])$ (and the same holds true, when exchanging $S$ with $T$). Since $\Reg(K[S])\subseteq \Reg(K[T])$ (Lemma \ref{lemma:conductor}), it follows that every prinicipal ideal of $\Reg(K[S])$ gets mapped to a principal ideal of $\Reg(K[T])$ by $\eta$. Now $\eta$ induces an epimorphism
\begin{align*}
\theta:\mathcal C_v(\Reg(K[S]))&\to\mathcal C_v(\Reg(K[T])),\\
[fK[X]\cap \Reg(K[S])]&\mapsto [fK[X]\cap \Reg(K[T])],
\end{align*}
whose kernel consists of those classes of $\mathcal C_v(\Reg(K[S]))$ containing $fK[X]\cap \Reg(K[S])$ for some $f\in \Reg(K[T])$. Since $\eta$ maps prime ideals to prime ideals, it follows that $\theta$ maps classes containing (infinitely many) prime divisors to classes containing (infinitely many) prime divisors. Now the assertion follows by \cite[Theorem 2.10.9.6]{GHK}.
\end{proof}

Let $K$ be a field and let $S$ be an affine monoid. Then the monoid algebra $K[S]$ is a Noetherian domain of finite Krull dimension. Moreover, $K[S]$ is a Krull domain if and only if $S$ is root closed and $K[S]$ is weakly Krull if and only if $S$ is a weakly Krull monoid. Numerical monoids are weakly Krull. By the class group of $K[S]$ we mean its $v$-class group.

\begin{theorem}\label{theorem:main1}
Let $K$ be a field and let $S$ be a numerical monoid with Frobenius number $\mathsf{f}(S)=\max(\N_0\setminus S)$. Then $K[S]$ has at least one prime divisor (resp. infinitely many) in all classes if for all $a_0,\ldots ,a_{\mathsf{f}(S)}\in K$ with $a_0\neq 0$ there exists at least one irreducible polynomial (resp. infinitely many) in $K[X]$ whose coefficient at the monomial $X^i$ equals $a_i$ for all $i\in[0,\mathsf{f}(S)]$.
\end{theorem}
\begin{proof}
Let $T$ be the numerical monoid with $\N_0\setminus T=[1,m]$ and apply Propositions \ref{proposition:classgroup1} and \ref{proposition:classgroup2}.
\end{proof}

We have already seen, that if $K$ is a finite field and $T$ is a numerical monoid, then $K[T]$ has prime divisors in all classes (it follows from the fact that $K[T]$ is an order in a holomorphy ring, see Remark \ref{remark:known} and the proof of Corollary \ref{corollary:irreducibles}). We give an alternative proof and extend the result to Hilbertian fields. Recall, that a field $K$ is said to be Hilbertian if every Hilbert set is non-empty, where a subset $H\subseteq K^n$ is called a Hilbert set if there exist polynomials $f_1,\ldots , f_l\in K[Y_1,\ldots ,Y_n,X_1,\ldots ,X_m]$ irreducible over $K(Y_1,\ldots ,Y_n)$ and a non-zero polynomial $g\in K[Y_1,\ldots ,Y_n]$ such that $H$ consists of those tuples $(y_1,\ldots ,y_n)\in K^n$ which satisfy $g(y_1,\ldots ,y_n)\neq 0$ and $f_j(y_1,\ldots ,y_n,X_1,\ldots ,X_m)$ are irreducible. The class of Hilbertian fields contains, for example, algebraic number fields and finitely generated transcendental extensions of arbitrary fields \cite[Theorem 13.4.2]{Fried}.

\begin{corollary}\label{corollary:main1}
Let $K$ be a Hilbertian field and let $S$ be a numerical monoid. Then $K[S]$ has infinitely many prime divisors in all classes.
\end{corollary}
\begin{proof}
By Theorem \ref{theorem:main1} it suffices to prove that $K[X]$ contains infinitely many irreducible polynomials for prescribed coefficients. If $K$ is finite, then this is \cite[Section 4.2]{Pollack}.\\
Let $K$ be a Hilbertian field, let $a_0,\ldots ,a_n\in K$ with $a_0\neq 0$ and let $m>n$. We set
\begin{align*}
f_m&=Y(a_0+a_1X+\ldots +a_nX^n)+X^m,\\
g&=Y.
\end{align*}
Then $f_m$ is irreducible over $K(Y)$ by Eisenstein's Criterion. Since every Hilbert set is non-empty, there exists $y\in K$ such that $y=g(y)\neq 0$ and $f_m(y,X)$ is irreducible over $K$, hence $a_0+a_1X+\ldots +a_nX^n+\frac{1}{y}X^m$ is irreducible over $K$.
\end{proof}

\begin{remark}\label{remark:examples}
Note that there are indeed affine monoid algebras which do not contain prime divisors in all classes. By Proposition \ref{proposition:classgroup1}, an example is $K[X^3,X^4,X^5]$ for $K$ an algebraically closed field. Observe that this ring is a non-principal order in the Dedekind domain $K[X]$, whence not Krull. This example can also be found in \cite[page 211]{Kain}. 
\end{remark}

\section{The higher-dimensional case} \label{section:higher-dim}

\begin{proposition}\label{proposition:use-kainrath}
Let $K$ be a field and let $S$ be an affine monoid with conductor $\mathfrak{f}_S = (S:\widehat{S})$. Assume that $S$ has a $v$-maximal ideal $\mathfrak{p}$ not containing $\mathfrak{f}_S$. Then the monoid algebra $K[S]$ has infinitely many prime divisors in each divisor class.
\end{proposition}

\begin{proof}
Since $S$ is torsion-free, it is an additive submonoid of the group $(\Z^m,+)$ for some $m \in \N$. Therefore the quotient group $G$ of $S$ is a $\Z$-submodule of the free $\Z$-module $\Z^m$ and hence free, say of rank $n \in \N$. Note that $n\geq 2$ or $K[S]$ is isomorphic to a (Laurent) polynomial ring in one variable over $K$. In the second case, the conclusion of the proposition is well-known \cite[Theorem 45.5]{Gilmer_multiplicative}. So we may assume that $n \geq 2$. \\
\textbf{Claim:} There exists a $\Z$-basis $B$ of $G$ such that $B \cap S \neq \emptyset$. \\
If $S = G$, this is trivial. So let $S \neq G$. Let $\mathfrak{p}\subseteq S$ be a $v$-maximal ideal of $S$ not containing $\mathfrak{f}_S$. Then by \cite[Theorem 2.6.5.3]{GHK}, the localization $S_\mathfrak{p}$ is a disrete valuation monoid. Let $\mathsf{v}_\mathfrak{p}: G \to \Z$ be the associated $\mathfrak{p}$-adic valuation. Let $a \in S$ with $\mathsf{v}_\mathfrak{p}(a) = 1$. Let $G_0$ be the kernel of $\mathsf{v}_\mathfrak{p}$. Then $G_0$ is a free $\Z$-module of rank $n-1$ and we have $G \cong \ker \mathsf{v}_\mathfrak{p} \oplus \im \mathsf{v}_\mathfrak{p} \cong G_0 \oplus \Z$. So, if $a_1,\ldots,a_{n-1}$ is a $\Z$-basis of $G_0$, then $a_1,\ldots,a_{n-1},a$ is a $\Z$-basis of $G$ with $a \in S$, which proves the Claim. \\
Now let $a_1,\ldots,a_n \in G$ be a $\Z$-basis with $a_1 \in S$. We define $\mathcal{O} = K[a_1]$ with quotient field $F = K(a_1)$. By $L = K(a_1,\ldots,a_n)$ we denote the quotient field of $K[S]$. Then $K[S]$ is a finitely generated $\mathcal{O}$-algebra, $F$ is a Hilbertian field (as a finitely generated transcendental extension of a field), the Krull dimension of $\mathcal{O}$ equals $1$, $\mathcal{O}$ is integrally closed and $\widehat{K[S]} = K[\widehat{S}]$ is a finitely generated $K[S]$-module (because $\widehat{S}$ is a finitely generated fractional ideal of $S$, see \cite[Prop. 2.7.4.2]{GHK}), $L/F$ is a purely transcendental extension and hence separabel and regular, and $\Pic(\mathcal{O}) = 0$. Therefore by \cite[Theorem 2]{Kain}, $K[S]$ has infinitely many prime divisors in all classes and we are done.
\end{proof}

\begin{lemma}\label{lemma:structure}
Let $S$ be an affine monoid with its complete integral closure $\widehat{S}$ being factorial.
\begin{enumerate}
\item $\widehat{S} \cong \N_0^s \oplus \Z^t$ for some non-negative integers $s$ and $t$.
\item $\mathfrak{X}(S) =  \{\mathfrak{p}_i \mid i \in [1,s]\}$, where $\mathfrak{p}_i = \{\alpha \in S \mid \pi_i(\alpha) \neq 0 \}$ and $\pi_i$ denotes the projection to the $i$-th coordinate of $S \subseteq \widehat{S} \cong \N_0^s \oplus \Z^t$ for $i \in [1,s]$.
\end{enumerate}
\end{lemma}

\begin{proof}
(1) $\widehat{S}$ is finitely generated by \cite[Theorem 2.7.14]{GHK}. That a finitely generated torsion free factorial monoid has the asserted form is clear. \\
(2) Since $S \subseteq \widehat{S}$ is a root extension \cite[Proposition 2.7.11]{GHK}, we have an inclusion preserving bijection $\spec(\widehat{S}) \to \spec(S), \mathfrak{q} \mapsto \mathfrak{q} \cap S$ \cite[Proposition 2.7]{Reinhart}. Thus, it suffices to prove the assertion for $S = \N_0^s \oplus \Z^t$. Since this monoid is factorial, $\mathfrak{X}(S) = \{p + S \mid p \in S \text{ is prime}\}$. The set of unit vectors $(\delta_{ij})_{j=1}^s$ for $i \in [1,s]$ clearly is a complete set of representatives of prime elements up to associates and the assertion follows.
\end{proof}

\begin{lemma}\label{lemma:monomials}
Let $K$ be a field and let $S$ be an affine monoid with complete integral closure $\widehat{S} =  \N_0^s \oplus \Z^t$ for some non-negative integers $s$ and $t$. Let $\mathfrak{f}_S=(S:\widehat{S})$ be the conductor of $S$ and denote $\mathsf{R} = \Reg(K[S])$.
\begin{enumerate}
\item $\widehat{K[S]} = K[\widehat{S}] = K[X_1,\ldots,X_s,Y_1,Y_1^{-1},\ldots,Y_t,Y_t^{-1}]$.
\item Let $f \in K[\widehat{S}]$ such that there exists $h \in K[\widehat{S}]$ with $fh \in \mathsf{R}$. Let $i \in [1,s]$ such that $\mathfrak{p}_i \supseteq \mathfrak{f}_S$ (cf. Lemma \ref{lemma:structure}(3)). Then $X_i$ does not divide $f$ in $K[\widehat{S}]$.
\end{enumerate}
\end{lemma}

\begin{proof}
(1) follows as in Lemma \ref{lemma:conductor}(1). \\
(2) Assume to the contrary that $X_i \mid f$ and let $g \in K[\widehat{S}]$ with $gX_i = f$. Since $\mathfrak{p}_i \supset \mathfrak{f}_S$, the ideal $K[\mathfrak{p}_i]$, which is a $v$-ideal by \cite[Lemma 2.3]{BIK} contains the conductor $\mathfrak{f} =K[\mathfrak{f}_S]$ of $K[S]$ (cf. Lemma \ref{lemma:conductor_general}). Now $K[S]$ is Noetherian \cite[Theorem 7.7]{Gil84}, whence Mori, so there exists a $v$-maximal ideal $Q$ of $K[S]$ such that $Q \supseteq K[\mathfrak{p}_i] \supseteq \mathfrak{f}$ \cite[Proposition 2.2.4.1]{GHK}.\\
By assumption, we have $X_igh = fh \in \mathsf{R} \subseteq K[S]$. It follows that $X_igh \in K[\mathfrak{p}_i]$ which is a contradiction to $\mathsf{R} \cap Q = \emptyset$ \cite[Theorem 2.10.9.2]{GHK}.
\end{proof}

\begin{lemma}\label{lemma:Reg}
Let $K$ be a field and let $S$ be an affine monoid with its complete integral closure being factorial. Denote $\mathsf{R} = \Reg(K[S])$. Let $f,g \in K[\widehat{S}]$ such that $fK[\widehat{S}] \cap \mathsf{R} \neq \emptyset$. 
\begin{enumerate}
\item If the coeffients of monomials of $f$ and $g$ having exponents in $S\setminus \mathfrak{f}_S$ are the same, then $\frac{f}{g} \in \mathsf{q}(\mathsf{R})$.
\item If $\frac{f}{g} \in \mathsf{q(R)}$, then $gK[\widehat{S}] \cap \mathsf{R} \neq \emptyset$ and $[fK[\widehat{S}] \cap \mathsf{R}] = [gK[\widehat{S}] \cap \mathsf{R}]$.
\end{enumerate}
\end{lemma}

\begin{proof}
(1) Let $f = f_1 + f_2$ where $f_1 \in K[S\setminus \mathfrak{f}_S]$ and $f_2  \in K[\mathfrak{f}_S]$ and $g = f_1 + g_2$ where $g_2 \in K[\mathfrak{f}_S]$. Since $fK[\widehat{S}] \cap \mathsf{R} \neq \emptyset$, we can pick $h \in K[\widehat{S}]$ with $hf \in \mathsf{R}$. Then, we have $hf_1 = hf - hf_2 \in K[S]$. \\
First, we want to show that $hf_1 \in \mathsf{R}$. Let $z \in K[\widehat{S}]$ such that $zhf_1 \in K[S]$. Then $zhf = zhf_1 + zhf_2 \in K[S]$ and since $hf \in \mathsf{R}$, it follows that $z \in K[S]$.\\
Now, it suffices to prove that $hg \in \mathsf{R}$. For this, note that $hg = hf_1 + hg_2 \in K[S]$ and let $z \in K[\widehat{S}]$ such that $zhg \in K[S]$. Then $zhf_1 = zhg - zhg_2 \in K[S]$. Since $hf_1 \in \mathsf{R}$, it follows that $z \in K[S]$.\\
(2) Since $\frac{f}{g} \in \mathsf{q(R)}$, we can pick $a,b \in \mathsf{R}$ such that $fb = ga$. Now, since $fK[\widehat{S}] \cap \mathsf{R} \neq \emptyset$, there exists $h \in K[\widehat{S}]$ such that $hf \in \mathsf{R}$. Therefore $gha = fhb \in \mathsf{R}$ and hence $gK[\widehat{S}] \cap \mathsf{R} \neq \emptyset$.\\
As noted in the preliminaries, we have an isomorphism
\begin{align*}
\delta: \mathcal{I}_\mathfrak{f}(K[\widehat S]) &\to \mathcal{I}_v^*(\mathsf{R}), \\
I &\mapsto I \cap \mathsf{R}.
\end{align*}
Now $aK[\widehat S] \cap \mathsf{R} = a\mathsf{R}$ and $bK[\widehat S] \cap \mathsf{R} = b\mathsf{R}$, because $\mathsf{R}$ is a saturated submonoid of $K[\widehat S]\setminus\{0\}$. So we have $b(fK[\widehat S] \cap \mathsf{R}) = \delta(bfK[\widehat S]) = \delta(agK[\widehat S]) = a(gK[\widehat S] \cap \mathsf{R})$.
\end{proof}

\begin{proposition}\label{proposition:factorial}
Let $K$ be a field and let $S$ be an affine monoid with factorial complete integral closure $\widehat{S}$, conductor $\mathfrak{f}_S$ and quotient group $G$ of rank at least $2$. If each height-one prime ideal of $S$ contains $\mathfrak{f}_S$, then the monoid algebra $K[S]$ has infinitely many prime divisors in each divisor class.
\end{proposition}

\begin{proof}
Without loss of generality, we may assume that $\widehat{S} = \N_0^s \oplus \Z^t$ for some non-negative integers $s$ and $t$ with $s + t \geq 2$ (see Lemma \ref{lemma:structure}). By \cite[Theorem 2.10.9.6]{GHK}, it suffices to prove the assertion for $\mathsf{R} = \Reg(K[S])$.\\
Let $I \in \mathcal{I}_v^*(\mathsf{R})$ and $f \in K[\widehat{S}] = K[X_1,\ldots,X_s,Y_1,Y_1^{-1},\ldots,Y_t,Y_t^{-1}]$ such that $fK[\widehat{S}] \cap \mathsf{R} = I$ \cite[Theorem 2.10.9.5]{GHK}. It suffices to construct infinitely many irreducible elements $g \in K[\widehat{S}]$ whose coefficients at monomials with exponents in $S \setminus \mathfrak{f}_S$ equal those of $f$. Then, by Lemma \ref{lemma:Reg}, $\frac{f}{g} \in \mathsf{q(R)}$ and hence $[I] = [gK[\widehat{S}] \cap \mathsf{R}]$. \\
We write $f = f_0 + f_1 X_1 + \ldots + f_n  X_1^n$ with $n \in  \N_0$ and $f_0, \ldots,f_n \in K[X_2,\ldots,X_s,Y_1,Y_1^{-1},\ldots,Y_t,Y_t^{-1}]$. Since $Y_1, \ldots,Y_t$ are units in $K[X_1,\ldots,X_s,Y_1,Y_1^{-1},\ldots,Y_t,Y_t^{-1}]$, we assume without loss of generality that $f_0, \ldots,f_n \in K[X_2,\ldots,X_s,Y_1,\ldots,Y_t]$. Let $\deg(f) < m \in \N_0$ such that $X_1^m \ldots X_s^m \in K[\mathfrak{f}_S]$ and $\text{char}(K)$ does not divide $m+1$. This is possible, because $\mathfrak{f}_S$ is a non-empty ideal of $\widehat{S}$ (see \cite[Theorem 2.7.13]{GHK}). We set 
\begin{align*}
G(T) = f + X_1^m \ldots X_s^m + T \cdot X_1^{m+1} X_2^m \ldots X_s^m,
\end{align*}
where $T$ is an indeterminate over $K(X_1,\ldots,X_s,Y_1,\ldots,Y_t)$. Since $G(T)$ is linear, it is irreducible in $K(X_1,\ldots,X_s,Y_1,\ldots,Y_t)[T]$ and by Lemma \ref{lemma:monomials}(2) it is primitive over $K[X_1,\ldots,X_s,Y_1,\ldots,Y_t]$. By the Lemma of Gauss, it follows that $G(T)$ is irreducible in $K[X_1,\ldots,X_s,Y_1,\ldots,Y_t][T]$ which is isomorphic to  $K[X_2,\ldots,X_s,Y_1,\ldots,Y_t,T][X_1]$ and hence over $K(X_2,\ldots,X_s,Y_1,\ldots,Y_t,T)$. \\
Since $K[X_2,\ldots,X_s,Y_1,\ldots,Y_t]$ is a Hilbertian ring \cite[Proposition 13.4.1]{Fried} and $G(T)$ is seperable over $K(X_2,\ldots,X_s,Y_1,\ldots,Y_t,T)$ ($\text{char}(K)$ does not divide $m+1$ and $G(T)$ is irreducible), by definition of a Hilbertian ring,  there exists some $a \in K[X_2,\ldots,X_s,Y_1,\ldots,Y_t]$ such that $g:=G(a)$ is irreducible over $K(X_2,\ldots,X_s,Y_1,\ldots,Y_t)$ \cite[Section 13.4]{Fried}. Again by Lemma \ref{lemma:monomials}(2), $g$ is a primitive polynomial over $K[X_2,\ldots,X_s,Y_1,\ldots,Y_t]$ and by the Lemma of Gauss $g$ is irreducible in $K[X_2,\ldots,X_s,Y_1,\ldots,Y_t][X_1] \cong K[X_1,\ldots,X_s,Y_1,\ldots,Y_t]$ and whence in $K[\widehat{S}]$. Moreover, the coefficients of $g$ coincide with those of $f$ at monomials with exponents in $S \setminus \mathfrak{f}_S$, because $X_1^m \ldots X_s^m  \in K[\mathfrak{f}_S]$ and $X_1^{m+1} X_2^m \ldots X_s^m  \in  K[\mathfrak{f}_S]$.
\end{proof}

\begin{theorem}
Let $K$ be a field and let $S$ be an affine monoid with factorial complete integral closure and quotient group of rank at least $2$. Then $K[S]$ has infinitely many prime divisors in each divisor class.
\end{theorem}

\begin{proof}
This follows by combining Propositions \ref{proposition:use-kainrath} and \ref{proposition:factorial}. To use Proposition \ref{proposition:use-kainrath}, note that every $v$-maximal ideal of $S$ not containing the conductor $\mathfrak{f}_S$ is of height one \cite[Theorem 2.6.5.3]{GHK}.
\end{proof}

\providecommand{\bysame}{\leavevmode\hbox to3em{\hrulefill}\thinspace}
\providecommand{\MR}{\relax\ifhmode\unskip\space\fi MR }
% \MRhref is called by the amsart/book/proc definition of \MR.
\providecommand{\MRhref}[2]{%
  \href{http://www.ams.org/mathscinet-getitem?mr=#1}{#2}
}
\providecommand{\href}[2]{#2}


\begin{thebibliography}{1}





\bibitem{Bourbaki}
N. Bourbaki, \emph{Commutative Algebra}, Springer-Verlag, 2nd printing, 1989.




\bibitem{Ch21a}
Gyu~Whan Chang, \emph{Every abelian group is the class group of a ring of
  {K}rull type}, J. Korean Math. Soc. \textbf{58} (2021), no.~1, 149--171.

\bibitem{Chou81}
L. Chouinard, \emph{Krull semigroups and divisor class groups}, Canad. J. Math. \textbf{33} (1981), 1459--1468.

\bibitem{BIK}
S. El Baghdadi, L. Izelgue and S. Kabbaj, \emph{On the class group of a graded domain}, J. of Pure and Appl. Algebra \textbf{171} (2002), 171--184.



\bibitem{FW}
V. Fadinger and D. Windisch, \emph{A characterization of weakly Krull monoid algebras}, \url{https://arxiv.org/abs/2009.04182}.

\bibitem{Krull}
V. Fadinger and D. Windisch, \emph{On the distribution of prime divisors in Krull monoid algebras}, \url{https://arxiv.org/abs/2101.04398}.

\bibitem{Fried}
M.D. Fried and M. Jarden, \emph{Field Arithmetic}, 3$^{\text{rd}}$ edition, Springer, 1986.



\bibitem{Ga-La-Sm19a}
A.~Garc{\'i}a-Elsener, P.~Lampe, and D.~Smertnig, \emph{Factoriality and class
  groups of cluster algebras}, Advances in Math. \textbf{358} (2019), {106858,
  48}.

\bibitem{GHK}
A.~Geroldinger and F.~Halter-Koch, \emph{Non-{U}nique {F}actorizations.
  {A}lgebraic, {C}ombinatorial and {A}nalytic {T}heory}, Pure and Applied
  Mathematics, vol. 278, Chapman \& Hall/CRC, 2006.


\bibitem{Gil84}
R. Gilmer, \emph{Commutative semigroup rings}, Chicago Lectures in Mathematics, 1984.

\bibitem{Gilmer_multiplicative}
R. Gilmer, \emph{Multiplicative Ideal Theory}, Marcel Dekker, 1972.

\bibitem{HK98}
F. Halter-Koch, \emph{IDEAL SYSTEMS: An Introduction to Multiplicative Ideal Theory}, Marcel Dekker, 1998.


\bibitem{Kain}
F. Kainrath, \emph{The distribution of prime divisors in finitely generated domains}, manuscripta math. \textbf{100} (1999), 203--212.





\bibitem{Pollack}
P. Pollack, \emph{Irreducible polynomials with several prescribed coefficients}, Finite Fields Appl. \textbf{22} (2013), 70--78.


\bibitem{Reinhart}
A. Reinhart, \emph{On integral domains that are C-monoids}, Houston J. Math. \textbf{39:4} (2013), 1095--1116.



\bibitem{Sm}
D.~Smertnig, \emph{Every abelian group is the class group of a simple
  {D}edekind domain}, Trans. Amer. Math. Soc. \textbf{369} (2017), 2477--2491.

\end{thebibliography}
\end{document}